\documentclass[12pt]{extarticle}
\usepackage{amsmath}
\usepackage{amsthm, amssymb, hyperref, color}
\pdfoutput=1
\usepackage{graphicx}
\usepackage{subfigure}
\usepackage{verbatim, mathrsfs}
\usepackage{blkarray}
\usepackage{color}
\oddsidemargin \evensidemargin
\date{}

\theoremstyle{definition}
\newtheorem{theorem}{Theorem}
\numberwithin{theorem}{section}
\newtheorem{proposition}[theorem]{Proposition}

\newtheorem{corollary}[theorem]{Corollary}
\newtheorem{definition}[theorem]{Definition}

\newtheorem{remark}[theorem]{Remark}
\newtheorem{example}[theorem]{Example}

\definecolor{g4}{rgb}{0,0.4,0}

\newcommand{\trace}[1]{\mbox{trace}{#1}}

\title{\bf Duality of Graphical Models \\ and Tensor Networks}
\author{Elina Robeva and Anna Seigal}
\date{}
\begin{document}
\maketitle

\begin{abstract} \noindent 
In this article we show the duality between tensor networks and undirected graphical models with discrete variables. We study tensor networks on hypergraphs, which we call tensor hypernetworks.
We show that the tensor hypernetwork on a hypergraph exactly corresponds to the graphical model given by the dual hypergraph.
We translate various notions under duality. For example, marginalization in a graphical model is dual to contraction in the tensor network. Algorithms also translate under duality. We show that belief propagation corresponds to a known algorithm for tensor network contraction. This article is a reminder that the research areas of graphical models and tensor networks can benefit from interaction.
\end{abstract}

\section{Introduction}

Graphical models and tensor networks are very popular but mostly separate fields of study. Graphical models are used in artificial intelligence, machine learning, and statistical mechanics \cite{WJ}. 
Tensor networks show up in areas such as quantum information theory and partial differential equations \cite{Hack,JM}. 
 
Tensor network states are tensors which factor according to the adjacency structure of the vertices of a graph. On the other hand, graphical models are probability distributions which factor according to the clique structure of a graph. The joint probability distribution of several discrete random variables is naturally organized into a tensor. Hence both graphical models and tensor networks are ways to represent families of tensors that factorize according to a graph structure.

The relationship between particular graphical models and particular tensor networks has been studied in the past. For example in \cite{CM} the authors reparametrize a hidden markov model to make a matrix product state tensor network. In \cite{CCXXW}, a map is constructed that sends a restricted Boltzmann machine graphical model to a matrix product state. In \cite{P}, an example of a directed graphical model is given with a related tensor network on the same graph, to highlight computational advantages of the graphical model in that setting.

From the outset, there are differences in the graphical description. On the graphical models side, the factors in the decomposition correspond to cliques in the graph. On the tensor networks side, the factors are associated to the vertices of the graph.

In this article, we show a duality correspondence between graphical models and tensor networks. This correspondence applies to all graphical models and all tensor networks and does not require reparametrization of either. Our mathematical relationship stems from hypergraph duality. We begin by recalling the definition of a hypergraph.

\begin{definition}
A hypergraph $H = (U,\mathcal{C})$ consists of a set of vertices $U$, and a set of hyperedges $\mathcal C$. A hyperedge $C\in\mathcal C$ is any subset of the vertices.
\end{definition}

There are two ways to construct a hypergraph from a matrix $M$ of size $d\times c$ with entries in $\{0,1\}$. First, we let the rows index the vertices and the columns index the hyperedges. The non-vanishing entries in each column give the vertices that appear in that hyperedge,
$$ M_{uC} = \begin{cases} 1 & u \in C \\ 0 & \text{otherwise.} \end{cases}$$
In this case $M$ is the {\em incidence matrix} of the hypergraph. We allow nested or repeated hyperedges, as well as edges containing one or no vertices, so there are no restrictions on $M$. Alternatively, we can construct a hypergraph with incidence matrix $M^T$. This is the {\em dual hypergraph} to the one with incidence matrix $M$, see \cite[Section 1.1]{B}.

We now add extra data to the matrix. We attach positive integers $n_1, \ldots, n_d$ to each row. We assign tensors to each column of $M$ whose size is the product of the $n_i$ as $i$ ranges over the non-vanishing entries in the column. For example, the tensor associated to the column $(1, 1, 0, 1, 0, \ldots, 0)^T$ would have size $n_1 \times n_2 \times n_4 $. We explain how this defines the data of both a graphical model and of a tensor network. Filling in the entries of the tensors gives a distribution in a graphical model (if we choose entries in $\mathbb{R}_{\geq 0}$), or a tensor network state in a tensor network. We see how a graphical model is visualized by the hypergraph with incidence matrix $M$, while the tensor network is visualized by the hypergraph of $M^T$.

Before stating our duality correspondence, we define graphical models in terms of hypergraphs, and introduce tensor hypernetworks. We keep in mind how the definitions translate to the incidence matrix set-up from above.

\begin{definition} Consider a hypergraph $H=(U, \mathcal C)$ with $U = [d]$. An {\em undirected graphical model} with respect to $H$ is the set of probability distributions on the random variables $\{X_u, u\in U\}$ which factor according to the hyperedges in $\mathcal C$:
$$P(x_1,\dots, x_d) = \frac1Z\prod_{C\in\mathcal C}\psi_C(x_C).$$
Here, the random variable $X_u$ takes values $x_u\in\mathcal X_u$, the subset $x_C$ equals $\{x_u: u\in C\}$, and the function $\psi_{C}$ is a {\em clique potential} with domain $\prod_{u\in C}\mathcal X_u$. The normalizing constant $Z$ ensures the probabilities sum to one.
\end{definition}

When all random variables are discrete, the joint probabilities form a tensor $P$ of size $\times_{u \in U} | \mathcal{X}_u |$ and the clique potentials are tensors of size $\times_{u \in C} | \mathcal{X}_u |$, all with entries in $\mathbb{R}_{\geq 0}$. The graphical model is depicted as the hypergraph whose incidence matrix has rows represented by the random variables $\{ X_u : 1 \leq u \leq d \}$ and columns indexed by the hyperedges.

If we fix the values in the clique potentials, we obtain a particular distribution in the graphical model. We recover the usual depiction of the graphical model by a graph instead of a hypergraph by connecting pairs of vertices by an edge if they lie in the same hyperedge. 

\begin{remark}
Graphical models are sometimes required to factorize according to the {\em maximal} cliques of a graph. We see later how our set-up specializes to this case. Models with cliques that are not necessarily maximal can be called {\em hierarchical models} \cite{Seth}.
\end{remark}

We now define tensor hypernetworks: families of tensors that factor according to a hypergraph. They have been defined this way in the literature
though it is not common ~\cite{BDR, BCM}.

\begin{definition}
Consider a hypergraph $G = (V,E)$. To each hyperedge $e\in E$ we associate a positive integer $n_e$, called the size of the hyperedge. To  each vertex $v \in V$ we assign a tensor $T_v\in \bigotimes_{e\ni v}\mathbb K^{n_e}$, where $\mathbb K$ is usually $\mathbb R$ or $\mathbb C$. The {\em tensor hypernetwork state} is obtained from $\bigotimes_{v \in V} T_v$ by contracting indices along all hyperedges in the graph that contain two or more vertices. We call hyperedges containing only one vertex dangling edges.
\end{definition}

Note that as opposed to graphical models, in tensor hypernetworks we assign tensors to the vertices of the graph rather than the hyperedges.

Restricting the definition of a tensor hypernetwork to hyperedges with at most two vertices gives the usual definition of a tensor network. The following example illustrates a widely used tensor network.

\begin{example}[Tucker decomposition] \label{tucker} Consider the graph
\begin{center}
\includegraphics[width=0.2\textwidth]{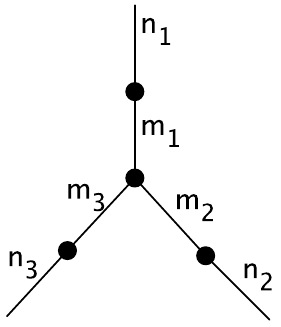}
\end{center}
We have a core tensor $T_0\in\mathbb K^{m_{1}}\otimes\cdots\otimes\mathbb K^{m_{d}}$, and matrices $T_i\in\mathbb K^{n_i}\otimes \mathbb K^{m_i}$. The entries of the tensor $T\in\mathbb K^{n_1}\otimes\cdots\mathbb K^{n_d}$ are
$$T_{i_1,\dots,i_d} = \sum_{j_1,\dots, j_d} (T_0)_{j_1,\dots, j_d}(T_1)_{i_1,j_1}\cdots(T_d)_{i_d, j_d}.$$
For suitable weights $m_i$ and orthogonal matrices $T_j$, this is the {\em Tucker decomposition} of $T$.
\end{example}

An important reason to extend the definition of tensor networks to tensor hypernetworks, other than the duality with graphical models explained in the next section, is that significant classes of tensors naturally arise from tensor hypernetworks.

\begin{example}[Tensor rank (CP rank)] \label{cp}
Consider this hypergraph on vertex set $\{ 1, 2, 3 \}$.
\begin{center}
\includegraphics[width=0.21\textwidth]{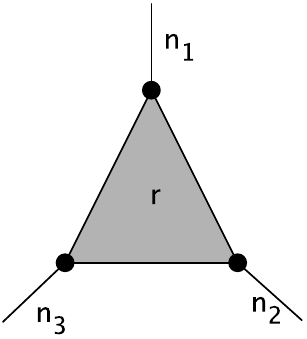}
\end{center}
There is one dangling edge for each vertex, with sizes $n_1, n_2, n_3$. There is one more hyperedge of size $r$, represented by a shaded triangle, that connects all three vertices. The tensors $T_v$ attached to each of the three vertices are matrices of size $n_v \times r$. The tensor hypernetwork state has size $n_1 \times n_2 \times n_3$ with entries 
$$ T_{i j k } = \sum_{l = 1}^{r} (T_1)_{il} (T_2)_{j l} (T_3)_{k l}. $$
The set of tensors given by this tensor hypernetwork equals the set of tensors of rank at most $r$. The same structure on $d$ vertices, with weights $n_1, \ldots, n_d, r$, gives tensors of size $n_1 \times \cdots \times n_d$, and rank at most $r$. Tensor rank is the most direct generalization of matrix rank to the setting of tensors \cite{Hack}. The set of tensors of rank at most r is naturally parametrized by this tensor hypernetwork without requiring special structure on the tensors at the vertices. \end{example}

The rest of the paper is organized as follows. We describe the duality correspondence between graphical models and tensor networks in Section~\ref{sec:duality}. In Section~\ref{properties} we explain how certain structures (graphs, trees, and homotopy types) and operators (marginalization, conditioning, and entropy) translate under the duality map. In Section~\ref{application} we give an algorithmic application of our duality correspondence.

\section{Duality} \label{sec:duality}

In this section we give the duality between graphical models and tensor networks.

\begin{theorem}\label{duality}
A discrete graphical model associated to a hypergraph $H = (U, \mathcal C)$ with clique potentials $\psi_C: \prod_{u\in C}\mathcal X_u \to \mathbb K$ is the same as the data of a tensor hypernetwork associated to its dual hypergraph $H^*$ with tensors $T_C = \psi_C$ at each vertex of $H^*$.
\end{theorem}

\begin{proof}
Consider a joint distribution (or tensor) $P$ in the graphical model defined by the hypergraph $H$. As described above, the incidence matrix $M$ of $H$ has rows corresponding to the variables $u \in U$ and columns corresponding to the cliques $C \in \mathcal{C}$. The data of the distribution $P$ also contains a potential function $\psi_C: \prod_{u\in C}\mathcal X_u \to\mathbb K$ for each clique $C\in\mathcal C$, which is equivalently a tensor of size $\times_{u \in C} |\mathcal{X}_u|$.

The dual hypergraph $H^*$ has incidence matrix $M^T$. It is a hypergraph with vertices $\{ v_C : C \in \mathcal{C} \}$ and hyperedges $\{ e_u : u \in U\}$.  By definition of the dual hypergraph, $u \in C$ is equivalent to $v_C \in e_u$. Associating the tensors $T_C = \psi_C\in\bigotimes_{e_u \ni v_C} \mathbb{K}^{|\mathcal X_u|}$ to each vertex $v_C$ of $H^*$ gives a tensor hypernetwork for $H^*$. Moreover, up to scaling by the normalization constant $Z$, the joint probability tensor $P$ is given by
$$P(x_u: u\in U) \cdot Z = \prod_{C\in\mathcal C} \psi_C(x_C) = \prod_{C\in\mathcal C}(T_C)_{x_C}.$$
The last expression is the tensor hypernetwork state before contracting the hyperedges. 
\end{proof}

Note that since $(M^T)^T = M$, the dual of the dual $(H^*)^*$ is equal to $H$. This implies the following one-to-one correspondence. Before we state it, let us denote  the set of distributions on $\mathcal{X} = \prod_{u \in U} \mathcal{X}_u$ that are in the graphical model defined by the hypergraph $H = (U,F)$ by $\mathscr{G}(H, \mathcal{X})$, and the set of non-contracted tensor hypernetwork states from a hypergraph $G = (V,E)$ with weights ${\bf n} = \{ n_e : e \in E\}$ by $\mathscr{T}(G, {\bf n})$.

\begin{corollary}
There is a one-to-one correspondence between the graphical models $\mathscr{G}(H, \mathcal{X} )$ and the tensor hypernetwork states $\mathscr{T}(H^*, \{|\mathcal X_u|:u\in U\})$ up to global scaling constant.
\end{corollary}

Note that while clique potentials are required to take values in $\mathbb{R}_{\geq 0}$ for probabilistic reasons, the definition and factorization structure of graphical models carries over to the case where the entries of these tensors belong to a general field $\mathbb{K}$. 
In the rest of this section we illustrate our results by showing the dual structures to some familiar examples of tensor network states and graphical models. 

\begin{example}[Matrix Product States (MPS)/Tensor Trains] \label{mps}
These are a popular family of tensor networks in quantum physics \cite{O} and numerical applications \cite{Hack} (where the two names come from). 
We return to them in detail in Section~\ref{application}. The MPS network on the left is dual to the graphical model on the right.

\begin{center}
\includegraphics[width=0.2\textwidth]{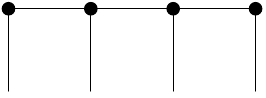}\hspace{1cm}\includegraphics[width=0.2\textwidth]{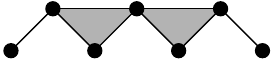}
\end{center}

The top row of edges in the tensor network is contracted. We see later that this corresponds to the top row of variables in the graphical model being hidden.
\end{example}

\begin{example}[No three-way interaction model]
This graphical model consists of all probability distributions that factor as
$ p_{ijk} = A_{ij} B_{ik} C_{jk} $,
for clique potential matrices $A, B, C$. It is represented by a hypergraph
in which all hyperedges have two vertices. The incidence matrix of the hypergraph is
$$
\begin{blockarray}{cccc}
& A & B & C \\
\begin{block}{c(ccc)}
  i & 1 & 1 & 0 \\
  j & 1 & 0 & 1  \\
  k & 0 & 1 & 1 \\
 \end{block}
\end{blockarray}
 $$
This matrix is symmetric. Hence the tensor network corresponding to this graphical model is given by the same triangle graph. We note that, up to dangling edges, this is also the shape of the tensor network of the tensor that represents the matrix multiplication operator~\cite{JM}.
\begin{center}
\includegraphics[width=0.2\textwidth]{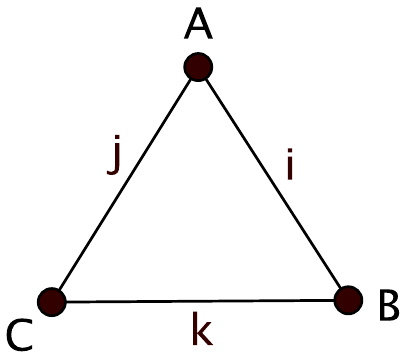}\hspace{1cm}
\includegraphics[width=0.2\textwidth]{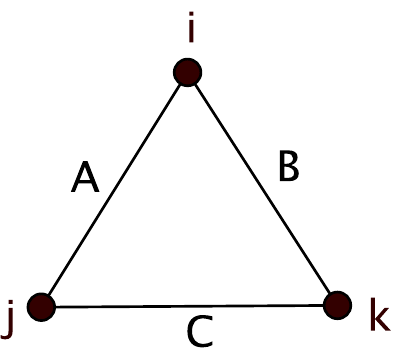}
\end{center}
\end{example}

\begin{example}[The Ising Model] \label{ising}
This graphical model is defined by the cliques of a two-dimensional lattice such as the grid on the right.
Its dual is the hypergraph on the left.
\begin{center}
\includegraphics[width=0.20\textwidth]{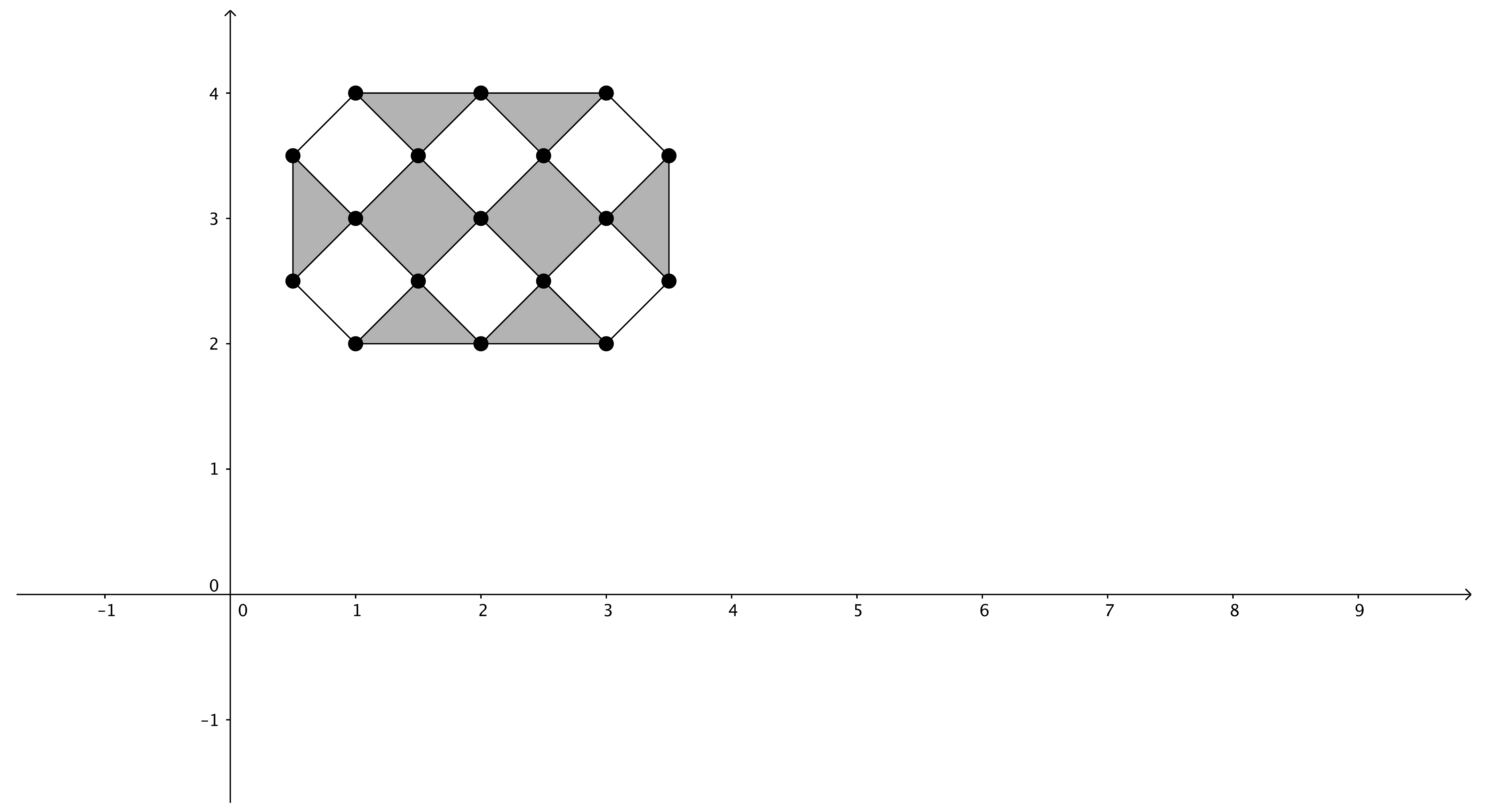}\hspace{1cm}
\includegraphics[width=0.20\textwidth]{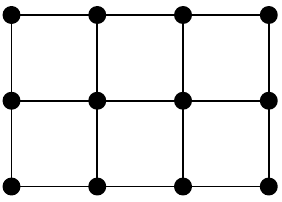}
\end{center}
\end{example}

\begin{example}[Projected Entangled Pair States (PEPS)] This tensor network is a two-dimensional analogue of MPS. It depicts two-dimensional quantum spin systems. Its hypergraph is depicted on the left, with its dual graphical model on the right. Note the structural similarity with Example~\ref{ising}.
\begin{center}
\includegraphics[width=0.20\textwidth]{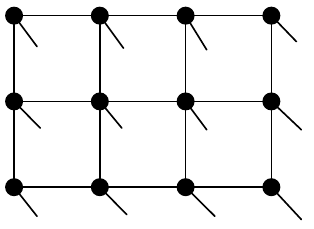}\hspace{1cm}
\includegraphics[width=0.20\textwidth]{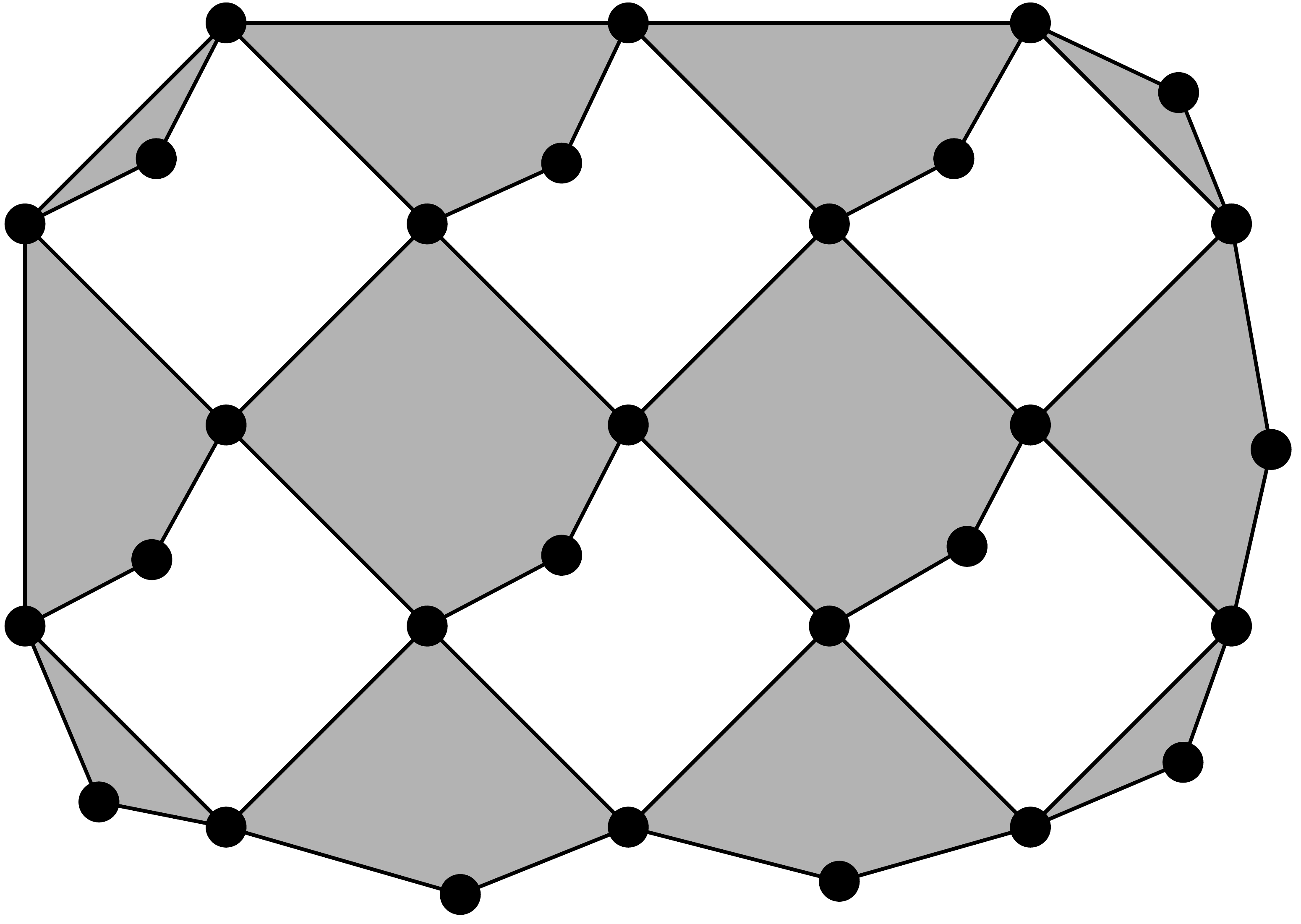}
\end{center}
\end{example}

\begin{example}[The Multi-scale Entanglement Renormalization Ansatz (MERA)]
This tensor network is popular in the quantum community, due to its favorable abilities to represent relevant tensors and compute efficiently with them. It is on the left, with its dual graphical model on the right. 
\begin{center}
\includegraphics[width=0.20\textwidth]{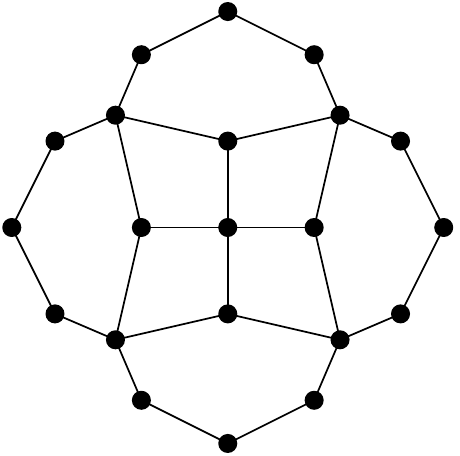}\hspace{1cm}
\includegraphics[width=0.20\textwidth]{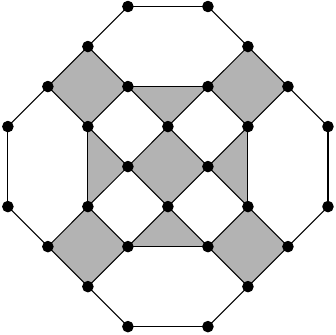}

\end{center}
\end{example}

Finally, we point out the following fun fact.
\begin{remark}[Duality of Tucker and CP decomposition]
Consider the hypergraphs in Examples~\ref{tucker} and \ref{cp} that give the graph structure of Tucker decomposition and CP decomposition respectively. Up to removal of dangling edges, the hypergraph corresponding to CP decomposition is dual to the one for Tucker decomposition.
\end{remark}

\section{Properties} \label{properties}

Tensor networks and graphical models are often given special structure.
For example, one can restrict to tensor networks that use a graph rather than a hypergraph.  In this section we show how properties and operations for graphical models and tensor hypernetworks behave under the duality map.

\subsection{Restricting to graphs}

Graphs are special hypergraphs in which every hyperedge contains two vertices. They are also known as $2$-uniform hypergraphs. Each column of the incidence matrix of such a hypergraph sums to two. Taking the dual of a graph gives a hypergraph in which every vertex has degree two, also known as a $2$-regular hypergraph \cite{B}. We call a hypergraph at-most-2-regular if every vertex has degree at most 2.

\begin{proposition}
Tensor networks are dual to at-most-$2$-regular graphical models. Graph models (graphical models whose cliques are the edges of a graph) are dual to $2$-regular tensor hypernetworks.
\end{proposition}

Graphical models defined by the maximal cliques of a graph correspond to hypergraphs in which we introduce a hyperedge for each maximal clique. Their dual tensor hypernetworks have the following property.

\begin{proposition}
Graphical models defined by the maximal cliques of a graph correspond to tensor hypernetworks whose hypergraphs have the property that whenever a set of hyperedges meet pairwise, the intersection of all of them is non-empty.
\end{proposition}

\begin{proof} Let $E'\subseteq E$ be a set of hyperedges that meet pairwise. Then, for all $e_1, e_2\in E'$, the corresponding vertices $u_{e_1}, u_{e_2}$ in the dual hypergraph (i.e. in the graphical model) are connected by an edge. Thus, the vertices $\{u_e:e\in E'\}$ form a clique in the graphical model, so there exists a maximal clique $C$ in which this clique is contained. Thus, all hyperedges in $E'$ contain the vertex corresponding to $C$.
\end{proof}

\subsection{Trees on each side}

The {\em homotopy type} of a hypergraph is the homotopy type of the simplicial complex whose maximal simplices are the maximal hyperedges. For topological purposes, we associate hypergraphs with their simplicial complexes. We show that the homotopy type of a hypergraph and its dual agree.

\begin{definition}[see \cite{H}] Consider an open cover $\mathcal{V} = \{ V_i : i \in I \}$ of a topological space $X$. The {\em nerve} $N(\mathcal{V})$ of the cover is a simplicial complex with one vertex for each open set. A subset $\{ V_j : j \in J \}$ spans a simplex in the nerve whenever $\cap_{j \in J} V_j \neq \emptyset$. 
\end{definition}

\begin{theorem}[The Nerve Lemma \cite{Bo}]
The homotopy type of a space $X$ equals the homotopy type of the nerve of an open cover of $X$, provided that all intersections $\cap_{j \in J} V_j$ of sets in the open cover are contractible.
\end{theorem}

We consider the open cover of our simplicial complex in which open sets are $\epsilon$-neighborhoods of the maximal simplices. For $\epsilon$ sufficiently small, such an open cover has contractible intersections, since they are homotopy equivalent to intersections of simplices. Hence the homotopy type of the hypergraph is equal to that of its nerve. The following proposition relates the nerve to the dual hypergraph.

\begin{proposition}
The nerve of a hypergraph is the simplicial complex of its dual hypergraph.
\end{proposition}

\begin{proof}
Consider a hypergraph $H$ with vertex set $U$ and hyperedge set $\mathcal{C}$. We now construct the dual hypergraph. The edges are represented by rows in the original incidence matrix. A subset $\{ C_j : j \in J\} \subseteq \mathcal{C}$ is connected by a hyperedge if there exists a vertex $u \in U$ that is in all hyperedges $C_j$ in the subset, or equivalently, if the intersection $\cap_{j \in J} C_j$ is non-empty. This is exactly the definition of the nerve.
\end{proof}

From this, the Nerve Lemma implies the following.

\begin{theorem} \label{cor:homotopy}
A tensor hypernetwork and its dual graphical model have the same homotopy type.
\end{theorem}

A hypergraph cycle (see \cite[Chapter 5]{B}) is a sequence $(x_1, E_1, x_2, E_2, x_3, \ldots, x_k , E_k, x_1)$, where the $E_i$ are distinct hyperedges and the $x_j$ are distinct vertices, such that $\{ x_i, x_{i+1} \} \subseteq E_i$ for all $i = 1, \ldots, k-1$, and $\{ x_1, x_k \} \subseteq E_k$. A tree is a hypergraph with no cycles. The simplicial complexes corresponding to trees are contractible. Theorem~\ref{cor:homotopy} implies that trees are preserved under the duality correspondence.

\subsection{Marginalization and contraction}
Let $H = (U, \mathcal C)$ be a hypergraph and $H^*$ its dual. Let $P$ be a distribution in the graphical model on $H$ with clique potentials $\psi_C:\prod_{u\in C}[n_u] \to \mathbb K$.
The dual tensor hypernetwork has tensors $T_C = \psi_C\in\bigotimes_{u\in C}\mathbb K^{n_u}$ at the vertices of $H^*$.

\begin{proposition}[Marginalization Equals Contraction] \label{marg} Let $W\subseteq U$ be a subset of the vertices of the graph $H$. Then, the marginal distribution of $\{X_u\}_{u\in W}$ equals
$$ P(x_W) = \sum_{x_u\in[n_u]:\atop{u\not\in W}}\prod_{C\in \mathcal C}(T_C)_{\{x_C: u\in C\}},$$
which is the contracted tensor hypernetwork along the hyperedges corresponding to $W^c$.
\end{proposition}
\begin{proof} The proof follows from the chain of equalities:
$$P(x_W) = \sum_{x_u\in[n_u]:\atop{u\not\in W}} P(x) = \sum_{x_u\in[n_u]:\atop{u\not\in W}}\prod_{C\in\mathcal C}\psi_C(x_C) = \sum_{x_u\in[n_u]:\atop{u\not\in W}}\prod_{C\in \mathcal C}(T_C)_{\{x_C: u\in C\}}.$$
In words, summing over the values of all variables in $W^c$ is the same as contracting the tensor hypernetwork along all hyperedges in $W^c$. 
\end{proof}

The interpretations of marginalization and contraction are also very similar in nature. The variables of a graphical model that are marginalized are often considered to be hidden, and the contracted edges of a tensor network represent entanglement (`unseen interaction').

The correspondence described in Proposition~\ref{marg} allows us to translate algorithms for marginalization in graphical models to algorithms for contraction in tensor networks, see Section \ref{application}. Without care to order indices, marginalization and contraction involve summing exponentially many terms. In many cases more efficient methods are possible.

\subsection{Conditional distributions}

Consider a probability distribution given by a fully-observed graphical model. Conditioning a variable $X_u$ to only take values in a given set $\mathcal Y_u\subseteq \mathcal X_u$ means restricting the probability tensor $P$ to the slice $\mathcal Y_u \times \prod_{b\in U\setminus\{u\}}\mathcal X_b$ which contains only the values $\mathcal Y_u$ for the variable $X_u$. This in turn corresponds to restricting each of the potentials for hyperedges $C$ containing $u$ to the given subset of elements $\mathcal Y_u\times \prod_{b\in C\setminus\{u\}}\mathcal X_b$. On the tensor networks side, we restrict the tensor corresponding to the given clique potential to the slice $\mathcal Y_u\times\prod_{b\in C\setminus\{u\}}\mathcal X_b$.

We wish to remark that the equivalence of conditioning and restriction to a slice of the probability tensor is due to the fact that the basis in which we view the probability tensor is fixed. The basis is given by the states of the the random variables: graphical models are not basis invariant. On the other hand, basis invariance is a key property of tensor networks that crops up in many applications, e.g. often a gauge (basis) is selected to make the computations efficient \cite{O}.

\subsection{Entanglement entropy and Shannon entropy}
Given a tensor network state represented by a tensor $T$, the {\em entanglement entropy}~\cite{O} equals $$-\trace(T\log T).$$ On the other hand, if $T$ represents the corresponding marginal distribution of the graphical model, the {\em Shannon entropy}~\cite{WJ} of $T$ is defined as
$$H(T) = - \sum_{\mathbf i\in\mathcal I}T_{\mathbf i}\log T_{\mathbf i},$$
where $\mathcal{I}$ indexes all entries of $T$. Expanding out the formula $-\trace(T\log T)$ shows that these two notions of entropy are the same. 

\section{Algorithms for marginalization and contraction} \label{application}

The {\em belief propagation} (or {\em sum-product}) algorithm is a dynamic programming method for computing marginals of a distribution~\cite{WJ}. The {\em junction tree algorithm} \cite{WJ} extends it to graphs with cycles. The equivalence between marginalization in graphical models and contraction in tensor hypernetworks was given in Proposition \ref{marg}. It means that we can use methods for marginalization to contract tensor hypernetworks and vice versa. For example, we can compute expectation values of tensor hypernetwork states~\cite{O} as well as contracted tensor hypernetwork states. In this section, we apply the junction tree algorithm to these tasks for the matrix product state (MPS) tensor networks from Example~\ref{mps}. We first recall the algorithm.

\subsection{The junction tree algorithm}

The input and output data of the junction tree algorithm are as follows.

\noindent{{\it Input:}} A graphical model defined by a hypergraph $H$ with clique potentials $\psi_C$.

\noindent{{\it Output:}} The marginals at each hyperedge, $P(x_C) = \sum_{x_u : u \notin C} P(x)$. 

We now recall how this algorithm works. First, we construct the graph $G$ associated to the hypergraph $H$ by adding edge $(i,j)$ whenever vertices $i$ and $j$ belong to the same hyperedge.  If $G$ is not chordal (or triangulated) we add edges until all cycles of length four or more have a chord, i.e. $G$ becomes chordal.
Then we can form the junction tree. This is a tree whose nodes are the maximal cliques of the graph. It has the {\em running intersection property}: the subset of cliques of $G$ containing a given vertex forms a connected subtree. Note that there are often multiple ways to construct a junction tree of a given graph $G$. 

To each maximal clique $C$ in $G$ we associate a clique potential which equals the product of the potentials of the hyperedges contained in $C$. If a hyperedge is contained in more than one maximal clique, its clique potential is assigned to one of them.
Each edge of the junction tree connects two cliques $C_1, C_2\in\mathcal C$ in $G$. We associate to such an edge the {\em separator} set $S = C_1\cap C_2$. We also assign a separator potential $\psi_S(x_S)$ to each $S$.  It is initialized to the constant value 1. A {\em basic message passing operation} from $C_1$ to a neighboring $C_2$ updates the potential functions at clique $C_2$ and separator $S = C_1\cap C_2$:
$$ \tilde \psi_S(x_S) \leftarrow \sum_{x_{C_1\setminus S}}\psi_{C_1}(x_{C_1}), $$
$$\hspace{1ex} \tilde\psi_{C_2}(x_{C_2}) \leftarrow \frac{\tilde \psi_S(x_S)}{\psi_{S}(x_S)}\psi_{C_2}(x_{C_2}).$$
The algorithm chooses a root of the junction tree, and orients all edges to point from the root outwards. It then applies basic message passing operations step-by-step from the root to the leaves until every node has received a message. Then the process is done in reverse, updating the clique and separator potentials from the leaves back to the root. After all messages have been passed, 
the final clique potentials
equal the marginals, $\tilde{\psi_C}  = \sum_{x_u \notin C}\prod_{B\in\mathcal C} \psi_{B}(x_B)$, and likewise for the final separator potentials.

\begin{remark} When the junction tree algorithm is used for probability distributions the clique potential functions are positive, but it works just as well for complex valued functions.
\end{remark}

The complexity of the junction tree algorithm is exponential in the {\em treewidth} of the graph, which is the size of the largest clique over all possible triangulations \cite[Chapter 2]{WJ}.

\subsection{Contracting a tensor network via duality}

To compute a tensor network state, we contract all edges in its tensor network that are not dangling. Our framework allows us to do this via duality, and to provably show the hardness of this computation since computing marginals on the graphical models side is widely studied \cite{WJ}.
The recipe is as follows. We consider the dual graphical model to the tensor hypernetwork. We make a new clique in the graphical model consisting of all vertices corresponding to the dangling edges of the tensor hypernetwork. The tensor hypernetwork state is the marginal distribution of that clique. We can then use, e.g., the junction tree algorithm to compute it.  

\subsection{Computing expectation values for matrix product states}

We now give the example of MPS tensor networks to illustrate how the junction tree algorithm translates to tensor hypernetworks. Using Theorem \ref{duality} we compute the family of graphical models that is dual to matrix product states. We show that the junction tree algorithm used to compute marginalizations of the dual graphical model corresponds to the {\em bubbling} algorithms that are used to compute expectation values of a MPS \cite{O}. In the figures, we draw the MPS with four observable indices, but repeating the pattern gives the results in the general case.

In quantum applications a tensor network state is denoted $| \psi \rangle$. Its expectation value is the inner product $\langle \psi | A | \psi \rangle$ for some operator $A$, which acts as a linear transformation in each vector space of observable indices (i.e. it is block diagonal).

Computing the expectation value of a MPS means contracting the tensor network on the left in Figure~\ref{MPS2}, where the middle row of vertices correspond to the blocks of $A$. Equivalently, it means marginalizing all variables of the graphical model on the right (or, computing the normalization constant of this graphical model). We contract the tensor network by applying the junction tree algorithm to the graphical model.

\begin{figure}[h]  
\begin{center}
\includegraphics[width=0.2\textwidth]{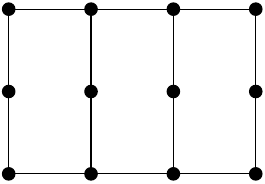} \hspace{1cm} \includegraphics[width=0.2\textwidth]{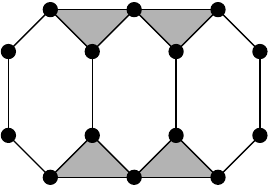}
\end{center}
\caption{The MPS tensor network on four states contracted with itself (left). Its dual graphical model (right).} \label{MPS2} 
\end{figure}

The first step of the algorithm is to triangulate the graph of the graphical model, by adding edges until it is chordal (or triangulated), see Figure~\ref{MPS3}. Next, we form a junction tree for the triangulated graph, see Figure~\ref{MPS4}.

\begin{figure}[h]
\begin{center}
\includegraphics[width=0.42\textwidth]{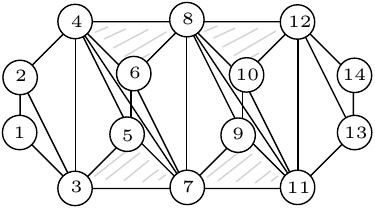}
\end{center}
\caption{A triangulation of the dual of a matrix product state.}
\label{MPS3}
\end{figure}

\begin{figure}[h]
\begin{center}
\includegraphics[width=\textwidth]{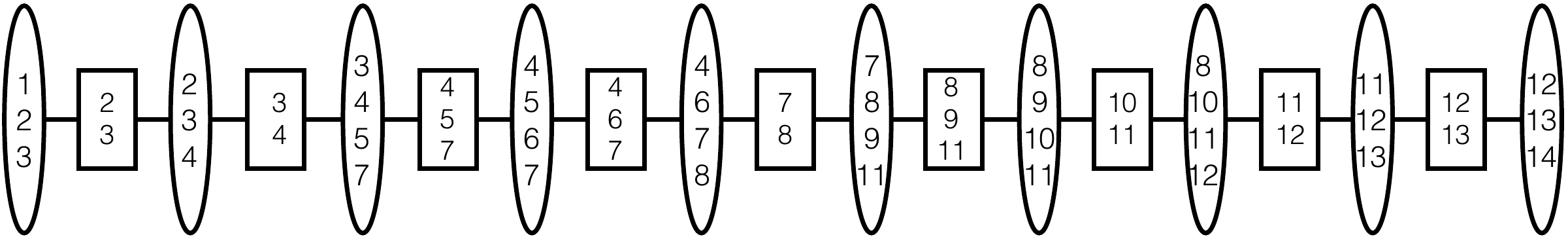}
\end{center}
\caption{The junction tree of the chordal graph in Figure~\ref{MPS3}. The cliques are in ovals; the separators are boxed.}
 \label{MPS4}
\end{figure}

We choose the root of the tree to be the left-most vertex in Figure~\ref{MPS4}. We do basic message passing operations from left to right until every vertex has received a message from its parent.
We arrive at the right-most clique $\{12, 13, 14\}$. If we complete the algorithm, by repeating this process from right to left, the final clique potentials at each vertex will equal the marginals. However, we can simplify the computation since our goal is just to compute the total sum. We terminate the message passing operations once we reach $\{12, 13, 14\}$. At that point we have the marginal at that clique, so we sum over the three vertices 12, 13, and 14 to get the total sum.

We now translate the junction tree algorithm to the language of tensor networks. The junction tree determines the order in which to contract the indices of the tensor network, see Figure~\ref{MPSContraction}. We contract edges in the tensor network until it is completely contracted.

At each step we sum over just one vertex of the dual graphical model (due to the structure of the junction tree in this case). This means means we contract one edge at a time from the tensor network.
In the first message passing operation we have $C_1 = \{ 1,2,3\}$, $C_2 = \{2,3,4\}$, $S=\{2,3\}$. We sum over the values of vertex $1$, since it is the only variable in $C_1 \backslash S$. This corresponds to contracting the tensor along the edge corresponding to vertex 1 of the graphical model (see step one of Figure~\ref{MPSContraction} for the corresponding tensor network operation). In the second message passing operation we sum over the values of vertex 2 of the graphical model. This corresponds to contracting the tensor network along the left edge (see the second step of Figure 4). The subsequent steps of the junction tree algorithm correspond to the steps shown in Figure 4.

\begin{figure}[h]
\begin{center}
\includegraphics[width = 0.99\textwidth]{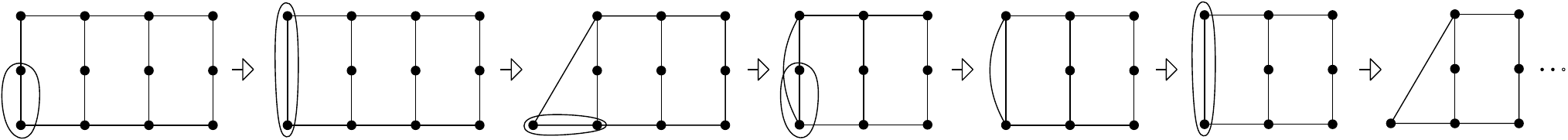}
\end{center}
\caption{Order of contraction in the MPS tensor network to compute its expectation value.}
\label{MPSContraction}
\end{figure}

It turns out that contracting the tensor in this way is what is usually done by the tensor networks community as well, a method sometimes called bubbling \cite{O}.
The triangulated graph of the dual graph of MPS has a treewidth of four, since we can continue the triangulation given in Figure~\ref{MPS3}. We can compute the complexity of the junction tree algorithm, and of the bubbling algorithm, to be $O(|V| (n r^3 + n^2 r^2 ))$ where $|V|$ is the number of vertices in the MPS, $n$ is the size of the dangling edges, and $r$ is the size of the entanglement edges.

\subsection{Extending to larger dimensions}

The higher-dimensional analogue of matrix product states/tensor trains is called the projected entangled pair states (PEPS), see Example~\ref{ising}. They are based on a two-dimensional lattice of entanglement interactions. Computing expectation values for the PEPS network takes exponential time in the number of states of the network \cite{O}. On the graphical models side, it is possible in principle to find expectation values of a PEPS state using the junction tree algorithm. Since the triangulated graph of the dual hypergraph of PEPS has a tree-width that grows in the size of the network, the junction tree algorithm is exponential time.

In \cite{MS}, the authors show that algorithms for computing expectation values are exponential in the treewidth of the tensor network. On the other hand, we have seen that the junction tree algorithm is exponential time in the treewidth of the dual graphical model. This indicates a similarity between the treewidth of a hypergraph and of its dual. A comparison of the treewidths of planar graphs and of their graph duals can be found in \cite{RS}.

To avoid exponential running times, numerical approximations are used. For graphical models, these are termed {\em loopy belief propagation} (see \cite[Chapter 4]{WJ} and references therein). A natural question is whether the algorithms for loopy belief propagation translate to known algorithms in the tensor networks community, e.g. for computing expectation values of PEPS, or whether they provide a new family of algorithms. In our opinion both answers to this question would be interesting.

\bigskip

{\bf Acknowledgements.} We would like to thank Jason Morton and Bernd Sturmfels for helpful discussions. Elina Robeva was funded by an NSF mathematical sciences postdoctoral research fellowship (DMS 1703821).

\begin{small}

\end{small}

\vspace{-.1cm}
\noindent \footnotesize {\bf Authors' addresses:}

\noindent 
Elina Robeva,  Massachusetts Institute of Technology, USA, 
{\tt erobeva@mit.edu}, \\
Anna Seigal,
University of California, Berkeley, USA,
{\tt seigal@berkeley.edu}.

\begin{thebibliography}{10}

\setlength{\itemsep}{-0.5mm}

\bibitem{BDR}R.~Bailly, F.~Denis, G.~Rabusseau: {\em Recognizable Series on Hypergraphs}, Language and Automata Theory and Applications, 639-651, 
Lecture Notes in Comput. Sci., 8977, Springer, Cham, 2015. 

\bibitem{BCM}A.~Banerjee, A.~Char, B.~Mondal: {\em Spectra of general hypergraphs}. Preprint arXiv:1601.02136.

\bibitem{B} C.~Berge: {\em Hypergraphs}, Combinatorics of finite sets, North-Holland Mathematical Library, 45. North-Holland Publishing Co., Amsterdam (1989). 

\bibitem{Bo} K.~Borsuk: {\em On the imbedding of systems of compacta in simplicial complexes}, Fund. Math. 35 (1948)
217-234.

\bibitem{CCXXW} J.~Chen, S.~Cheng, H.~Xie, L.~Wang, T.~Xiang: {\em On the Equivalence of Restricted Boltzmann Machines and Tensor Network States}, preprint, arXiv:1701.04831 (2017).

\bibitem{CM} A.~Critch, J.~Morton: {\em Algebraic Geometry of Matrix Product States},  SIGMA Symmetry Integrability Geom. Methods Appl. 10 (2014).

\bibitem{Hack} W.~Hackbusch: {\em Tensor spaces and numerical tensor calculus}, Springer Series in Computational Mathematics, 42. Springer, Heidelberg (2012).

\bibitem{H} A.~Hatcher: {\em Algebraic topology}, Cambridge University Press, Cambridge (2002).

\bibitem{Lauritzen}S.L.~Lauritzen: {\em Graphical Models}, Oxford Statistical Science Series, 17. Oxford Science Publications. The Clarendon Press, Oxford University Press, New York (1996).

\bibitem{JM} J.~M.~Lansberg: {\em Tensors and their uses in Approximation Theory, Quantum Information Theory and Geometry}, draft notes (2017).

\bibitem{MS} I.L.~Markov, Y.~Shi: {\em Simulating quantum computation by contracting tensor networks}, SIAM J. Comput. 38 (2008), no. 3, 963-981. 

\bibitem{O} R.~Or\'us: {\em A practical introduction to tensor networks: matrix product states and projected entangled pair states}, Ann. Physics 349 (2014), 117--158. 

\bibitem{P} M.~Pejic: {\em Quantum Bayesian networks with application to games displaying Parrondo's paradox}, Thesis (Ph.D.), University of California, Berkeley (2014).

\bibitem{RS} N.~Robertson, P.D.~Seymour: {\em Graph minors. III. Planar tree-width}, J. Combin. Theory Ser. B 36 (1984), no. 1, 49-64.

\bibitem{Seth} S.~Sullivant: {\em Algebraic Statistics}, draft copy of book to appear (2017).

\bibitem{WJ}M.~Wainwright  and M.~I.~Jordan: {\em Graphical Models, Exponential Families, and Variational Inference}, Foundation and Trends in Machine Learning, vol 1, nos 1-2 (2008).

 \end{thebibliography}
\end{document}